\documentclass[12pt,reqno]{amsart}
\usepackage[numbers, square]{natbib}
\usepackage{mathptmx}
\usepackage{amsmath}
\usepackage{amscd}
\usepackage{amssymb}
\usepackage{amsthm}
\usepackage{enumerate}
\usepackage{xspace}
\usepackage[all,tips]{xy}
\usepackage[dvips]{graphicx}
\usepackage{verbatim}
\usepackage{syntonly}
\usepackage{hyperref}
\usepackage{amsmath, amsthm, graphics, amssymb,fullpage,color, epsfig,url}
\usepackage{indentfirst}
\usepackage{color}

\theoremstyle{plain}
\newtheorem{thm}{Theorem}
\newtheorem{lem}{Lemma}
\newtheorem{prop}{Proposition}
\newtheorem{defn}{Definition}
\newtheorem{cor}{Corollary}

\newtheorem{clm}{Claim}

\theoremstyle{definition}
\newtheorem*{rem}{Remark}

\newcommand{\disp}{\displaystyle}

\DeclareMathOperator{\di}{div}
\DeclareMathOperator{\curl}{curl}

\DeclareMathOperator{\tr}{tr}

\newcommand{\eps}{\varepsilon}
\newcommand{\vp}{\varphi}

\newcommand{\al}{\alpha}
\newcommand{\be}{\beta}
\newcommand{\ga}{\gamma}
\newcommand{\de}{\delta}

\newcommand{\te}{\theta}
\newcommand{\la}{\lambda}

\newcommand{\om}{\omega}

\newcommand{\nid}{\noindent}

\newcommand{\iny}{\infty}
\newcommand{\del}{ \partial}
\newcommand{\su}{\subset}
\newcommand{\LP}{\Delta}
\newcommand{\gr}{\nabla}

\newcommand{\norm}[1]{\left\| #1\right\|}

\newcommand{\abs}[1]{\left\vert#1\right\vert}
\newcommand{\set}[1]{\left\{#1\right\}}
\newcommand{\brac}[1]{\left[#1\right]}
\newcommand{\pr}[1]{\left( #1 \right) }
\newcommand{\pb}[1]{\left( #1 \right] }

\newcommand{\der}[2]{\frac{\del #1}{\del #2}}

\newcommand{\N}{\ensuremath{\mathbb{N}}}

\newcommand{\R}{\ensuremath{\mathbb{R}}}

\newcommand{\C}{\ensuremath{\mathbb{C}}}

\newcommand{\Keywords}[1]{\par\noindent 
{\small{\bf Keywords\/}: #1}}
\newcommand{\MSC}[1]{\par\noindent 
{\small{\bf Mathematics Subject Classification\/}: #1}}

\title{On Landis' conjecture in the plane when the potential \\ has an exponentially decaying negative part}

\author[Davey]{Blair Davey}
\address{Department of Mathematics, City College of New York CUNY, New York, NY 10031, USA}
\email{bdavey@ccny.cuny.edu}
\thanks{Davey is supported in part by the Simons Foundation Grant 430198.}
\author[Kenig]{Carlos Kenig}
\address{Department of Mathematics, University of Chicago, Chicago, IL 60637, USA}
\email{cek@math.uchicago.edu}
\thanks{Kenig is supported in part by NSF DMS-1265249.}

\author[Wang]{Jenn-Nan Wang}
\address{Institute of Applied Mathematical Sciences, NCTS, National Taiwan University,
Taipei 106, Taiwan}
\email{jnwang@math.ntu.edu.tw}
\thanks{Wang is supported in part by MOST 105-2115-M-002-014-MY3.}

\date{}

\begin{document}

\maketitle

\begin{abstract}
In this article, we continue our investigation into the unique continuation properties of real-valued solutions to elliptic equations in the plane.
More precisely, we make another step towards proving a quantitative version of Landis' conjecture by establishing unique continuation at infinity estimates for solutions to equations of the form $- \LP u +  V u = 0$ in $\R^2$, where $V = V_+ - V_-$, $V_+ \in L^\iny$, and $V_-$ is a non-trivial function that exhibits exponential decay at infinity.
The main tool in the proof of this theorem is an order of vanishing estimate in combination with an iteration scheme.
To prove the order of vanishing estimate, we establish a similarity principle for vector-valued Beltrami systems.
 \\

 \\

\MSC{35B60, 35J10} 
\end{abstract}

\section{Introduction}

In this paper, we consider the unique continuation properties of real-valued solutions to equations of the form
\begin{align}
- \LP u +  V u = 0
\label{ePDE}
\end{align}
in $\R^2$.
We assume that $V = V_+ - V_-$ where $V_\pm \ge 0$ satisfies
\begin{align}
& \norm{V_+}_{L^\iny\pr{\R^2} } \le 1
\label{V+Cond} \\
& V_-\pr{z} \le \exp\pr{- c_0 \abs{z}^{1 + \eps_0}} \quad \forall z \in \R^2,
\label{V-Cond}
\end{align}
for some $\eps_0 > 0$.
The main result of this article is the following quantitative form of Landis' conjecture for solutions to \eqref{ePDE}.

\begin{thm}
Assume that $V : \R^2 \to \R$ satisfies \eqref{V+Cond} and \eqref{V-Cond}.
Let $u: \R^2 \to \R$ be a solution to \eqref{ePDE} for which
\begin{align}
& \abs{u\pr{z}} \le \exp\pr{C_0 \abs{z}}
\label{uBd} \\
& \abs{u\pr{0}} \ge 1.
\label{normed}
\end{align}
Then for any $\epsilon > 0$ and any $R \ge R_0\pr{C_0, c_0, \eps_0, \eps}$, we have
\begin{equation}
\inf_{\abs{z_0} = R} \norm{u}_{L^\iny\pr{B_1\pr{z_0}}} \ge \exp\pr{- R^{1+\eps}}.
\label{globalEst}
\end{equation}
\label{LandisThm}
\end{thm}

This theorem improves upon the work in \cite{KSW15} (see also the subsequent results in \cite{DKW17} and \cite{DW17}) since we now allow for $V_-$ to be a non-trivial function.

To prove Theorem \ref{LandisThm}, we follow the usual approach and prove an order of vanishing estimate for a scaled version of equation \eqref{ePDE}.
Since the potential function exhibits decay at infinity, we combine the scaling argument first developed in \cite{BK05} with an iteration scheme similar to the one presented in \cite{Dav14} (and further developed in \cite{LW14}) to prove Theorem \ref{LandisThm}.

The notation $B_r\pr{z_0}$ is used to denote the ball of radius $r$ centered at $z_0 \in \R^2$.
The abbreviated notation $B_r$ will be used when the centre is understood from the context.
We also use the notation $Q_r\pr{z_0}$ to denote the cube of sidelength $2r$ centered at $z_0 \in \R^2$, and we may abbreviate the notation when it is clear from the context.
For the order of vanishing estimate, we consider solutions to \eqref{ePDE} in $Q_b$ for some $b > 1$.

\begin{thm}
Let $F$ be a function for which $1 \le F\pr{\la} \le \la$ for all $\la \ge 1$.
For some $\la \ge 1$, set $b = 1 + \frac 1 {F\pr{\la}}$. 
Assume that $\norm{V_+}_{L^\iny\pr{Q_b}} \le \la^2$ and that $\norm{V_-}_{L^\iny\pr{Q_b}} \le \de^2$, where 
\begin{equation}
\de = \frac{c_0 \sqrt \la}{\log \la} \exp\pr{-m \la}
\label{deBd}
\end{equation} 
for a constant $m > 0$ to be specified below.
Let $u$ be a real-valued solution to \eqref{ePDE} in $Q_b$ that satisfies, for some $p >0$,
\begin{align}
& \norm{u}_{L^\iny\pr{B_b}} \le \exp\pr{C_1 \la}
\label{localBd} \\
& \norm{u}_{L^\iny\pr{B_1}} \ge \exp\pr{- c_1 \la^p}.
\label{localNorm}
\end{align}
Then for any $r$ sufficiently small, 
\begin{equation}
\norm{u}_{L^\iny\pr{B_r}} \ge r^{C \la^q F\pr{\la}},
\label{localEst}
\end{equation}
where $q = \max\set{1, p}$ and $C$ depends on $c_0$, $C_1$, and $c_1$.
\label{OofV}
\end{thm}

Since we are working with real-valued solutions and equations in the plane, we follow an approach that is based on the ideas first developed in \cite{KSW15}.
In particular, we rely on tools from complex analysis to prove our theorem.
In \cite{KSW15}, \cite{DKW17}, and \cite{DW17}, the first step in the proof of the order of vanishing estimate is to show the existence of a positive multiplier and establish good bounds for it.
Since the negative part of $V$ is now assumed to be non-trivial, our usual approach to establishing the existence of a positive multiplier breaks down.
Thus, we introduce a positive solution to an associated equation with a shifted potential function.
This positive function allows us to transform the PDE for $u$ into a divergence-form equation.
The resulting equation is not divergence-free, but it resembles a higher-dimensional divergence-free equation.
Therefore, we mimic ideas from the $3$-dimensional setting, and introduce a vector-valued stream function that gives rise to a vector-valued Beltrami system.

The main challenge that we overcome is understanding the quantitative behavior of solutions to vector-valued Beltrami equations.
In the scalar setting, an application of the similarity principle in combination with the Hadamard three-circle theorem allowed us to quantify all solutions to the resulting Beltrami system.
As a similarity principle with bounds was not available to us in the vector-valued setting, we prove one here using Cartan's Lemma, the Wiener-Masani Theorem, and the ideas from \cite{BR03}.
With this new similarity principle, we can prove our three-ball inequalities by applying the Hadamard three-circle theorem component-wise.

Each section in this article describes an important proof.
Section \ref{S3} gives the proof of Theorem \ref{OofV} where each major step is presented in a subsection.
The four steps in this proof are: the introduction of a positive multiplier and its properties, the reduction from the PDE to a vector-valued Beltrami system, the quantitative properties of solutions to vector-valued Beltrami systems, and the three-ball inequality.
The proof of Theorem \ref{LandisThm} is presented in Section \ref{Landis}.
We first present the proposition behind the iteration scheme, whose proof relies on the order of vanishing estimate given in Theorem \ref{OofV}. 
Then we repeatedly apply the proposition to prove Theorem \ref{LandisThm}.
Finally, Section \ref{P1Proof} presents the proof of an important proposition in the quantification of solutions to vector-valued Beltrami systems.

\section{The proof of Theorem \ref{OofV}}
\label{S3}

\subsection{The positive multiplier}

In \cite{KSW15} and \cite{DKW17}, the first step in the proofs of the order of vanishing estimates is to establish that a positive multiplier associated to the operator (or its adjoint) exists and has suitable bounds.
Since we are no longer working with a zeroth order term that is assumed to be non-negative, we  take a somewhat different approach here. 

Define
$$V_\de\pr{x, y} = V\pr{x, y} + \de^2.$$
It follows from the assumptions $\norm{V_-}_{L^\iny\pr{Q_b}} \le \de^2$, $\norm{V_+}_{L^\iny\pr{Q_b}} \le \la^2$, and $\la \ge 1 \ge \de$, that $0 \le V_\de \le 2 \la^2$ a.e. in $Q_b$.
Therefore, we may mimic the techniques from \cite{KSW15} and \cite{DKW17} to construct a positive multiplier associated with the equation
\begin{equation}
\LP \phi - V_\de \phi = 0 \quad \text{ in } Q_b.
\label{localEPDE2}
\end{equation}
Set $\phi_1\pr{x,y} = \exp\pr{\sqrt{2}\la x}$.  
Since 
$$\LP \phi_1 - V_\de \phi_1 = \pr{2\la^2 - V_\de} \phi_1 \ge 0,$$ 
then $\phi_1$ is a subsolution.
Set $\phi_2 = \exp\pr{\sqrt{8}\la}$ and notice that 
$$\LP \phi_2 - V_\de \phi_2 = - V_\de \phi_2 \le 0,$$ 
so $\phi_2$ is a supersolution.
Since $\phi_2 \ge \phi_1$ in $Q_b$, then there exists a positive solution $\phi$ to \eqref{localEPDE2} for which
\begin{equation}
\exp\pr{- \sqrt{8} \la} \le \phi \le \exp\pr{\sqrt{8} \la}
\label{phiBounds}
\end{equation}
in $Q_b$.
By the gradient estimate for Poisson's equation (as in \cite{GT01} for example), we have
\begin{align}
&\norm{\gr \phi}_{L^\iny\pr{B_r}} \le \frac{C_\al \la^2}{r} \norm{\phi}_{L^\iny\pr{B_{\al r}}},
\label{gradphiEst}
\end{align}
whenever $\al > 1$, $\al r < b$.
Note that $C_\al \sim \pr{\al -1}^{-1}$.
A similar estimate holds for $u$ as well.

We present an estimate similar to one in \cite{KSW15} that will be instrumental below.

\begin{lem}
Recall that $b  = 1 + \frac 1 {F\pr{\la}}$, where $1 \le F\pr{\la} \le \la$.
For $d =  1 + \frac 1 {2 F\pr{\la}}$, there is an absolute constant $C_2$ for which
$$\norm{\gr \pr{\log \phi}}_{L^\iny\pr{Q_{d}}} \le C_2 \la,$$
where $\phi$ is a positive solution to \eqref{localEPDE2}.
\label{grphiLem}
\end{lem}

\begin{proof}
We begin with an $L^2$ estimate for $\Phi := \log \phi$.
Let $\te \in C_0^\iny\pr{Q_b}$ be a smooth cutoff function with $\te \equiv 1$ in $Q_{\tilde d}$, where $\tilde d = 1 + \frac 3 {4F\pr{\la}}$.
The assumption on $b$ implies that $b - \tilde d \ge \frac 1 {4F\pr{\la}}$ and therefore $\norm{\gr \te}_{L^\iny\pr{Q_b}} \le C F\pr{\la}$ and $\norm{\LP \te}_{L^\iny\pr{Q_b}} \le C \brac{F\pr{\la}}^2$.
It follows from \eqref{localEPDE2} that in $Q_b$
\begin{align}
\LP \Phi + \abs{\gr \Phi}^2 = V_\de.
\label{PhiPDE}
\end{align}
Multiplying both sides of this equation by $\te^2$ and integrating by parts, we see that
\begin{align*}
\int \abs{\gr \Phi}^2\te^2 
&= \int V_\de \te^2 + \int \gr\pr{ \te^2} \cdot \gr\Phi
\le \int V_\de \te^2 + \frac 1 2 \int \abs{\gr\Phi}^2 \te^2 + 2 \int \abs{\gr \te}^2.
\end{align*}
Therefore,
\begin{align*}
\int_{Q_{\tilde d}} \abs{\gr \Phi}^2 
&\le \int \abs{\gr \Phi}^2\te^2 
\le 2 \int V_\de \te^2 + 4 \int \abs{\gr \te}^2
\le C \pr{\la^2 + \brac{F\pr{\la}}^{2}},
\end{align*}
where we have used the bound on $V_\de$ and that $\tilde d \le \frac 7 4$.
Since $F\pr{\la} \le \la$, then $\norm{\gr \Phi}_{L^2\pr{Q_{\tilde d}}} \le C \la$, where $C$ is an absolute constant.

We rescale equation \eqref{PhiPDE}.
Set $\vp = \frac{\Phi}{C \la}$ for some $C > 0$.
Then \eqref{PhiPDE} is equivalent to 
\begin{equation}
\mu \LP \vp + \abs{\gr \vp}^2 = \widetilde V \,\,\, \text{ in } Q_d,
\label{vpPDE}
\end{equation}
where $\mu = \frac{1}{C \la}$ and $\widetilde V = \frac{V_\de}{C^2 \la^2}$.
Now choose $C$ sufficiently large so that
\begin{align}
\norm{\widetilde V}_{L^\iny\pr{Q_b}} \le 1, \quad 
\int_{Q_{\tilde d}} \abs{\gr \vp}^2 \le 1.
\label{scaleBds}
\end{align}

\begin{clm}
For any $z \in Q_{d}$, if $\mu \le c r$, $r < \frac 1 {8F\pr{\la}}$, and conditions \eqref{vpPDE} and \eqref{scaleBds} hold, then
$$\int_{B_r\pr{z}} \abs{\gr \vp}^2 \le C r^2.$$
\label{clm1}
\end{clm}

\begin{proof}[Proof of Claim \ref{clm1}]
We use the abbreviated notation $B_r$ to denote $B_r\pr{z}$ for some $z \in Q_d$.
Let $\eta \in C^\iny_0\pr{B_{2r}}$ be a cutoff function such that $\eta \equiv 1$ in $B_r$.
By the divergence theorem
\begin{align}
0 &= \mu \int \di \pr{\gr \vp \, \eta^2} 
= \mu \int \LP \vp \eta^2 + 2 \mu \int \eta  \gr \vp \cdot \gr \eta. 
\label{DivThmRes}
\end{align}
Now we estimate each term.
By \eqref{vpPDE} and \eqref{scaleBds},
\begin{align}
\int \mu \LP \vp \, \eta^2 
&= - \int \abs{\gr \vp}^2 \eta^2 + \int \widetilde V \eta^2
\le - \int \abs{\gr \vp}^2 \eta^2 + \norm{\widetilde V}_{L^\iny\pr{B_d}} \int_{B_{2r}} 1 \nonumber \\
&\le - \int \abs{\gr \vp}^2 \eta^2 + C r^2.
\label{IBd}
\end{align}
By Cauchy-Schwarz and Young's inequality,
\begin{align}
\abs{2 \mu \int \eta \gr \vp \cdot \gr \eta} 
&\le 2 \mu \pr{ \int \abs{\gr \vp}^2 \eta^2 }^{1/2} \pr{ \int \abs{\gr \eta}^2 }^{1/2} 
\le \frac{1}{2} \int \abs{\gr \vp}^2 \eta^2 + C \mu^2.
\label{IIBd}
\end{align}
Combining \eqref{DivThmRes}-\eqref{IIBd} and using that $\mu \le c r$, we see that
\begin{align}
\int_{B_r} \abs{\gr \vp}^2 
\le \int \abs{\gr \vp}^2 \eta^2
&\le C \mu^2 + C r^2 
\le C r^2, 
\label{combinedEst}
\end{align}
proving the claim.
\end{proof}

We now use Claim \ref{clm1} to give a pointwise bound for $\gr \vp$ in $Q_{d}$.
Define 
$$\vp_\mu\pr{z} = \frac{1}{\mu} \vp\pr{\mu z} .$$
Then
$$\gr \vp_\mu\pr{z} = \gr \vp \pr{\mu z}, \qquad  
\LP \vp_\mu\pr{z} = \mu \LP \vp \pr{\mu z}.$$
It follows from \eqref{vpPDE} that
\begin{equation*}
\LP \vp_\mu + \abs{\gr \vp_\mu}^2 = \widetilde V\pr{\mu z} : = \widetilde V_\mu\pr{z},
\end{equation*}
$$\norm{\widetilde V_\mu}_{L^\iny\pr{B_1}} \le 1.$$
Moreover,
\begin{align*}
\int_{B_{2}}\abs{\gr \vp\pr{\mu z}}^2 
= \frac{1}{\mu^2} \int_{B_{2 \mu}} \abs{\gr \vp}^2
\le \frac{1}{\mu^2} C \pr{2\mu}^2 = C,
\end{align*}
where we have used Claim \ref{clm1}.
It follows from Theorem 2.3 and Proposition 2.1 in Chapter V of \cite{Gi83} that there exists $p > 2$ such that
\begin{equation}
\norm{\gr \vp_\mu}_{L^p\pr{B_1}} \le C.
\label{QiaBd}
\end{equation}
Define
$$\tilde \vp_\mu \pr{z} = \vp_\mu\pr{z} - \frac{1}{\abs{B_1}} \int_{B_1} \vp_\mu.$$
Since $\gr \tilde \vp_\mu = \gr \vp_\mu$, then 
$$\LP \tilde \vp_\mu = - \abs{\gr \tilde \vp_\mu}^2 + \widetilde V_\mu : = \zeta \quad \text{ in } B_1.$$
Clearly, $\norm{\zeta}_{L^{p/2}\pr{B_1}} \le C$.
Moreover, by H\"older, Poincar\'e and \eqref{QiaBd},
$$\norm{\tilde \vp_\mu}_{L^{p/2}\pr{B_1}} 
\le C \norm{\tilde \vp_\mu}_{L^{p}\pr{B_1}} 
\le C \norm{\gr \tilde \vp_\mu}_{L^p\pr{B_1}} \le C.$$
By Theorem 9.9 from \cite{GT01}, for example,
$$\norm{\tilde \vp_\mu}_{W^{2, p/2}\pr{B_r}} \le C,$$
for any $r < 1$.
If $p > 4$, then it follows that $\norm{\gr \tilde \vp_\mu}_{L^\iny\pr{B_{r^\prime}}} \le C$.
Otherwise, assuming that $p < 4$, a Sobolev embedding shows that $\norm{\gr \tilde \vp_\mu}_{L^{\frac{2p}{4-p}}\pr{B_r}} \le C$.
Since $\frac{2p}{4-p} > p$, we may repeat these arguments to show that for some $r^\prime < 1$
$$ \norm{\gr \vp}_{L^\iny\pr{B_{ \mu r^\prime}}}  = \norm{\gr \vp_\mu}_{L^\iny\pr{B_{r^\prime}}}  =  \norm{\gr \tilde \vp_\mu}_{L^\iny\pr{B_{r^\prime}}} \le C.$$
This derivation works for any $z \in Q_{d}$ and any $\mu < \mu_0$. 
Since $\vp = \frac{\Phi}{C \la}$, the conclusion of the lemma follows.
\end{proof}

\subsection{Reduction to a vector-valued Beltrami equation}
Now we use the positive multiplier $\phi$ from above to reduce the PDE to a first-order Beltrami equation.
The novelty here is that the resulting equation is a vector equation instead of a scalar equation as it was in \cite{KSW15}.
With $u$ and $\phi$ satisfying \eqref{ePDE} and \eqref{localEPDE2}, respectively, we define
$$ v = \frac{u}{\phi},$$ 
and a computation shows that
\begin{equation}
\gr \cdot \pr{\phi^2 \gr v} + \de^2 \phi^2 v = 0 \quad \textrm{ in } Q_b.
\label{gradEq}
\end{equation}

\begin{defn}
For any $\de \in \R$, define the operator $\gr_\de = \pr{\del_x, \del_y, \de}$, where $\de$ denotes multiplication by $\de$.
That is, if $f$ is an arbitrary scalar function and ${\bf F} = \pr{F_1, F_2, F_3}$ is an arbitrary vector function, then 
\begin{align*}
&\gr_\de f = \pr{\del_x f, \del_y f, \de f} \\
&\gr_\de \cdot {\bf F} = \di_\de {\bf F} = \del_x F_1 + \del_y F_2 + \de F_3 \\
&\gr_\de \times {\bf F} = \curl_\de {\bf F} = \pr{\del_y F_3 - \de F_2, \de F_1 - \del_x F_3, \del_x F_2 - \del_y F_1}
\end{align*}
\end{defn}

With this new notation, \eqref{gradEq} may be rewritten as
\begin{equation}
\gr_\de \cdot \pr{\phi^2 \gr_\de v} = 0 \quad \textrm{ in } Q_b.
\label{lagradEq}
\end{equation}
Therefore, the positive multiplier $\phi$ for the related equation \eqref{ePDE} has been used to transform the PDE \eqref{ePDE} into a $\de$-divergence-free equation.

If we take the standard gradient, divergence and curl in $\R^3$, and replace $\del_z$ with multiplication by the constant $\de$, we get the operators $\gr_\de$, $\gr_\de \cdot$ and $\gr_\de \times$.
A number of the relationships between gradient, divergence and curl are inherited for these new operators.  
For example, $\gr_\de \cdot \pr{\gr_\de \times {\bf F} }= 0$, $\gr_\de \times \gr_\de f= \bf 0$, and $\gr_\de \times \pr{\gr_\de \times {\bf F} }= -\pr{\LP + \de^2}{\bf F} + \gr_\de \pr{\gr_\de \cdot {\bf F}}.$

The next step is to generalize the definition of the stream function given in \cite{KSW15}.
Since we have a $\de$-divergence-free vector field, $\phi^2 \gr_\de v$, the idea (that comes from the 3-dimensional setting) is to define a vector-valued function $\bf G$ that satisfies 
\begin{equation}
\gr_\de \times {\bf G} = \phi^2 \gr_\de v .
\end{equation}
That is, if ${\bf G} = \pr{v_1, v_2, v_3}$, then
\begin{equation}
\left\{\begin{array}{rl}
\del_y v_3 -  \de v_2 &=  \phi^2 \del_x v   \\
-\del_x v_3 + \de v_1 &=  \phi^2 \del_y v   \\
\del_x v_2 - \del_y v_1 &= \de \phi^2 v  \\
\end{array}  \right..
\label{streamSyst}
\end{equation}
Note that when $\de = 0$, this system reduces to the defining equations for the scalar stream function $v_3$.
When $\de \ne 0$, one possible solution to this system is obtained by setting $v_3 = 0$. 
That is, 
\begin{equation}
\left\{\begin{array}{rl}
 v_1 &:=  \de^{-1} \phi^2 \del_y v   \\
 v_2 &:=  - \de^{-1} \phi^2 \del_x v   \\
\end{array}  \right..
\label{vDef}
\end{equation}
Define
\begin{align}
\left\{\begin{array}{l}
w_1 := \phi^2 v \\
w_2 := v_2 + i  v_1
\end{array}\right..
\label{wDef}
\end{align}
With $\bar \del = \frac{\del}{\del \bar z} = \frac{1}{2}\pr{ \der{}{x} + i \der{}{y} }$, 
\begin{align*}
\bar \del w_1 
&= \bar\del\pr{\phi^2} v + \phi^2 \bar\del v
= 2 \bar\del\pr{\log \phi}  \phi^2 v + \frac{1}{2}\brac{\phi^2 \der{v}{x} + i \phi^2 \der{v}{y}} \\
&= 2 \bar\del\pr{\log \phi}  w_1 - \frac{\de}{2}\overline w_2,
\end{align*}
and
\begin{align*}
\bar \del w_2 
&=\bar \del v_2 + i \bar \del v_1
= \frac{1}{2}\brac{\der{v_2}{x} - \der{v_1}{y}} + \frac{i}{2}\brac{ \der{v_2}{y} + \der{v_1}{x}} \\
&= \frac{\de}{2} w_1 + \bar \del \pr{\log \phi} w_2 - \del \pr{\log \phi} \overline w_2.
\end{align*}
Set $\al = \bar{\del} \pr{\log \phi}$ so that $\overline \al = \overline{\bar{\del} \pr{\log \phi}} = \del \pr{\log \phi}$, since $\phi$ is real. 
We define
$$\widetilde \al = \left\{\begin{array}{ll} \overline{ \al} \frac{\overline w_2}{w_2} & \text{if } w_2 \ne 0 \\  0 & \text{otherwise}   \end{array}\right., \qquad
\widetilde \de = \left\{\begin{array}{ll} \de \frac{\overline w_2}{w_2} & \text{if } w_2 \ne 0 \\  0 & \text{otherwise}   \end{array}\right..$$
%then
%\begin{align*}
%& \bar \del w_1 -2 \al  w_1 =  - \frac{\widetilde \de}{2} w_2 \\
%& \bar \del w_2 - \pr{\al - \widetilde \al} w_2 = \frac{\de}{2} w_1.
%\end{align*}
We use the notation $T = T_{Q_b}$ to denote the Cauchy-Pompeiu operator on $Q_b$.
More details can be found in the next subsection, but for now we rely on the property that $\bar \del T_{Q_b} f = f \chi_{Q_b}$.
It follows that
\begin{align*}
& \bar \del \pr{e^{-T\pr{2 \al}} w_1} 
= e^{-T\pr{2 \al}}\bar \del w_1 -2 \al  e^{-T\pr{2 \al}} w_1 
=  - \frac{\widetilde \de}{2} e^{-T\pr{2 \al}}  w_2 \\
& \bar \del  \pr{e^{-T\pr{\al - \widetilde \al}} w_2} 
= e^{-T\pr{\al - \widetilde \al}} \bar \del w_2 - \pr{\al - \widetilde \al} e^{-T\pr{\al - \widetilde \al}} w_2 = \frac{\de}{2} e^{-T\pr{\al - \widetilde \al}} w_1.
\end{align*}
If we set $\widetilde w_1 = e^{-T\pr{2 \al}} w_1$, $\widetilde w_2 = e^{-T\pr{\al - \widetilde \al}} w_2$, %then
%\begin{align*}
%& \bar \del \widetilde w_1 =  - \frac{\widetilde \de}{2} e^{-T\pr{ \al+ \widetilde \al}}  \widetilde w_2 \\
%& \bar \del \widetilde w_2 = \frac{\de}{2} e^{T\pr{\al + \widetilde \al}} \widetilde w_1 .
%\end{align*}
and introduce vector notation
$$\vec{w} = \brac{\begin{array}{c} \widetilde w_1 \\ \widetilde w_2 \end{array}} \quad \text{ and } \quad 
G = \brac{\begin{array}{cc} 0 & -\frac{\widetilde \de}{2} e^{-T\pr{ \al+ \widetilde \al}} \\ \frac{\de}{2} e^{T\pr{\al + \widetilde \al}} & 0 \end{array}},$$
then we have
\begin{equation}
\bar \del \vec{w} - G \vec{w} = \vec{0} \quad \text{ in }\, Q_d.
\label{vecEqn}
\end{equation}

\subsection{Solutions to Beltrami matrix equations}
Towards understanding the behavior of solutions to \eqref{vecEqn}, we study the behavior of matrix solutions to the equation
\begin{equation}
\bar \del P - A P = 0 \quad \text{ in }\, \mathcal{R} := \brac{0, 1} \times \brac{0, 1}.
\label{matEqn}
\end{equation}

For a $2 \times 2$ matrix $A$, recall that 
$$\abs{A}^2 = \tr\pr{A^* A},$$
where $A^*$ is the Hermitian adjoint of $A$.
We use the notation $\norm{\cdot}$ to denote the operator norm of a matrix.
Observe that for $2 \times 2$ matrices $A$ and $B$, 
\begin{align*}
&\abs{AB} \le \abs{A} \norm{B} \\
&\norm{A} \le \abs{A} \le \sqrt 2 \norm{A} \\
&\abs{I} = \sqrt 2.
\end{align*}
For a $2 \times 2$ matrix function $A$, we write
$$\norm{A}_\iny = \sup_{i, j = 1,2}\norm{a_{ij}}_{L^\iny}.$$
The goal is to solve the equation $\bar \del P = A P$ in $\mathcal{R}$ and show that both $P$ and $P^{-1}$ have good control in terms of $M = \norm{A}_\iny$.

We first need some notation.
For some $\de > 0$, set 
$$V_i = \brac{i\de, i \de + \frac 3 2 \de}.$$
Then $V_{i-1} \cap V_i = \brac{i \de, i\de + \frac 1 2 \de}$ and $V_i \cap V_j \ne \emptyset$ if and only if $j = i \pm 1$.
Assuming that $\de$ is chosen so that $i_0 := \frac 1 \de - \frac 3 2 \in \N$, we have $\disp \brac{0, 1} = \bigcup_{i=1}^{i_0} V_i$.
Define 
$$U_i = V_i \times \brac{0, 1}.$$

The first proposition serves as the main tool in the proof the second proposition.

\begin{prop}
\label{matrixEstProp}
Let $\disp \set{H_i}_{i = 1}^{i_0}$ be a collection of $2 \times 2$ matrices such that each $H_i$ is defined on $U_{i-1} \cap U_{i}$, $\norm{H_i} \le 10$, $\norm{H_i^{-1}} \le 10$, and both $H_i$ and $H_i^{-1}$ are analytic on $U_{i-1} \cap U_{i}$.
Then there exists a collection of $2 \times 2$ matrices $\set{g_i}_{i = 0}^{i_0}$, where both $g_i$ and $g_i^{-1}$ are defined and analytic on $U_i$, with $H_i = g_{i-1} g_i^{-1}$ on $U_{i-1} \cap U_i$.
Moreover, there exists a constant $C > 0$ so that 
\begin{equation}
\label{gBds}
\abs{g_i}^2 + \abs{g_i^{-1}}^2 \le C e^{C/\de^2} \; \text{ in } \; U_i.
\end{equation}
\end{prop}

The proof of Proposition \ref{matrixEstProp} can be found in Section \ref{P1Proof}.
Here we use the result to prove the following proposition.

\begin{prop}
\label{matrixSolProp}
Let $A$ be a $2 \times 2$ matrix function defined on $\mathcal{R}$ with $M = \norm{A}_\iny$.
There exists an invertible solution to $\bar \del P = A P$ in $\mathcal{R}$ with the property that 
\begin{equation}
\label{PBds}
\norm{P} + \norm{P^{-1}} \le \exp\brac{C M^2 \pr{\log M}^2}.
\end{equation}
\end{prop}

\begin{proof}
For a constant $C_1 > 0$ to be specified below, define $\de$ so that $\de \log\pr{1/\de} \le \frac 1 {3 C_1 M}$.
In particular, if $M \ge M_0$, then there exists $c_1$ depending on $M_0$ and $C_1$ so that if
\begin{equation}
\label{deDef}
\de := \frac{c_1}{M \log M},
\end{equation} 
then the bound above is satisfied and $i_0 = \frac 1 \de - \frac 3 2 \in \N$.

We first solve the equation $\bar \del P = A P$ in $R_\de := U_i$. \\
If $P$ a solution and $P = I + Q$, then 
$$\bar \del Q = \bar \del P = A P = A + A Q.$$
Let 
$$ T_{R_\de}\pr{F}\pr{z} = \frac 1 \pi \int_{R_\de} \frac{F\pr{\xi}}{z - \xi} d\om\pr{\xi}.$$
Note that $\bar \del\pr{T_{R_\de}\pr{F}} = F \chi_{R_\de}$.
Since we need to solve the equation $\bar \del Q - A Q = A$ in $R_\de$, we solve $\bar \del\brac{Q - T_{R_\de}\pr{AQ} - T_{R_\de}\pr{A}} = 0$.
Therefore, we seek solutions to 
\begin{equation}
\label{RdeEq}
Q - T_{R_\de}\pr{A Q} = T_{R_\de}\pr{A} \; \text{ in } \; R_\de. 
\end{equation}

\nid{\em Observation:} There exists a constant $C>0$ such that $\disp \sup_{z \in R_\de} \int_{R_\de} \frac 1 {\abs{z - \xi}} d\om\pr{\xi} \le C \de \log \pr{1/\de}$. \\
Recall that $R_\de = U_i = \brac{i \de, i \de + \frac 3 2 \de} \times \brac{0, 1}$.
Partition $\brac{0, 1}$ into equal intervals $I_k$ of length at most $\frac 3 2 \de$.
Then $\disp U_i = \bigcup_{k=1}^{\lceil 2/3\de\rceil} \brac{i\de, i \de + \frac 3 2 \de} \times I_k$.
Assume first that $z \in \brac{i \de, i \de + \frac 3 2 \de} \times I_1$.
For $k \ge 3$, if $\xi \in \brac{i \de, i \de + \frac 3 2 \de} \times I_k$, then $\abs{z - \xi} \simeq k \de$, and 
$$ \int_{\brac{i \de, i \de + \frac 3 2 \de} \times I_k} \frac 1 {\abs{z - \xi}} d\om\pr{\xi} \le \frac 1 {k \de} \de^2 = \frac \de k.$$
Moreover, 
$$\int_{\brac{i \de, i \de + \frac 3 2 \de} \times \pr{I_1 \cup I_2}} \frac 1 {\abs{z - \xi}} d\om\pr{\xi} \lesssim \int_{B\pr{z, 5\de}} \frac 1 {\abs{z - \xi}} d\om\pr{\xi} \simeq \int_0^{5\de} \frac r r dr \simeq \de.$$
Hence $\disp \int_{R_\de} \frac{1}{\abs{z - \xi}} d\om\pr{\xi} \lesssim \sum_{k=1}^{\lceil 2/3\de\rceil} \frac \de k \simeq \de \pr{\log\pr{1/\de}}$.
When $z \in \brac{i \de, i \de + \frac 3 2 \de} \times I_{k_0}$ for $k_0 > 1$, the result follows similarly and we have proved the observation. \\

\nid{\em Claim:} There exists a constant $C_1 > 0$ so that $\disp \norm{T_{R_\de}\pr{F}}_{L^\iny\pr{R_\de}} \le C_1 \de \log\pr{1/\de} \norm{F}_\iny$. \\
This claim follows directly from the observation above and the definition of the operator $T_{R_\de}$.

By the definition of $\de$ given in \eqref{deDef}, we have that $C_1 \de \log\pr{1/\de} M \le 1/3$.
Therefore, we can solve \eqref{RdeEq} via a Neumann series approach.
Moreover, the resulting solution $Q$ has $\norm{Q}_\iny \le \frac 3 2 C_1\de \log\pr{1/\de} M \le \frac 1 2$ and then $P = I + Q$ satisfies $\norm{P} < 3$ and $\norm{P^{-1}} < 3$.

Using the construction described above, for each $i = 0, \ldots, i_0$, define $P_i$ to be the matrix solution to 
$$\bar \del P_i= A P_i \; \text{ in } \; U_i$$ 
with $\norm{P_i} < 3$ and $\norm{P_i^{-1}} < 3$.
On $U_{i-1} \cap U_i$, define $H_i = P_{i-1}^{-1} P_i$.
Clearly, $\norm{H_i} \le 10$ and $\norm{H_i^{-1}} \le 10$.
As 
$$\bar \del H_i = - P_{i-1}^{-1} \bar \del P_{i-1} P_{i-1}^{-1} P_i + P_{i-1}^{-1} \bar \del P_i = - P_{i-1}^{-1} A P_{i-1} P_{i-1}^{-1} P_i + P_{i-1}^{-1} A P_i = 0,$$
then each $H_i$ is analytic on $U_{i-1} \cap U_i$.
A similar argument shows that each $H_i^{-1}$ is also analytic on $U_{i-1} \cap U_i$.
Therefore, Proposition \ref{matrixEstProp} is applicable.
That is, there exist functions $g_i$ defined and analytic on $U_i$ such that $H_i = g_{i-1} g_i^{-1}$ on $U_{i-1} \cap U_i$ and $\abs{g_i}^2 \le C e^{C/\de^2}$ on $U_i$.

Now we use the collections $\set{P_i}_{i=1}^{i_0}$ and $\set{g_i}_{i=1}^{i_0}$ to define a function $P$ on all of $R$.
On $U_i$, set $P = P_i g_i$.
Since each $g_i$ is analytic, then $\bar \del P = \bar \del P_i g_i = A P_i g_i = A P$ on each $U_i$, as required.
As $H_i = P_{i-1}^{-1} P_i = g_{i-1} g_i^{-1}$, then $P_i g_i = P_{i-1} g_{i-1}$ on $U_{i-1} \cap U_i$.
Moreover, 
$$\norm{P} + \norm{P^{-1}} = \norm{P_i g_i} + \norm{g_i^{-1} P_i^{-1}} \lesssim e^{C/\de^2}.$$
Referring to \eqref{deDef}, the estimate \eqref{PBds} follows.
\end{proof}

\begin{rem}
Although this construction was done on the unit rectangle (for convenience), since $d \in \brac{1, 3/2}$, the result still holds with a modified constant when $\mathcal{R}$ is replaced by $Q_d$.
\end{rem}

\begin{lem}
Let $\disp c_\iny = \sup_{s \in \brac{1, 3/2}}\set{\norm{T_{Q_{s}}}_{L^\iny\pr{Q_{s}} \to L^\iny\pr{Q_{s}}}}$.
Let $\disp C_3 =  \sup_{s \in \brac{1, 3/2}}\set{C_2\pr{s}}$, where $C_2\pr{s}$ is the constant given in Lemma \ref{grphiLem} on $Q_s$.
If we set $m = 2 c_\iny C_3$, then the matrix $G$ belongs to $L^\iny\pr{Q_d}$ and satisfies
$$\norm{G}_{L^\iny\pr{Q_d}} \le \frac{C\sqrt\la}{\log \la}.$$
\label{GBd}
\end{lem}

\begin{proof}
Recall that $\disp G = \brac{\begin{array}{cc} 0 & -\frac{\widetilde \de}{2} e^{-T\pr{ \al+ \widetilde \al}} \\ \frac{\de}{2} e^{T\pr{\al + \widetilde \al}} & 0 \end{array}}$.
Since $\al = \bar{\del} \pr{\log \phi}$, $\overline \al = \del \pr{\log \phi}$, $\abs{\widetilde \al} = \abs{\bar \al}$, and $d \in \brac{1, 3/2}$, then it follows from Lemma \ref{grphiLem} that $\norm{\al}_{L^\iny\pr{Q_d}} \le C_3 \la$ and $\norm{\tilde \al}_{L^\iny\pr{Q_d}} \le C_3 \la$.
Therefore, $\norm{T\pr{\al + \widetilde \al}}_{L^\iny\pr{Q_d}} \le 2 c_\iny C_3 \la$ and then
\begin{align*}
\norm{G}_{L^\iny\pr{Q_d}} 
&\le \frac{\de}{2} \exp\pr{2 c_\iny C_3 \la} 
\le \frac{c_0 \sqrt\la}{2\log \la} \exp\pr{-m \la} \exp\pr{2 c_\iny C_3 \la} 
= \frac{C\sqrt\la}{\log \la},
\end{align*}
where we have used \eqref{deBd}.
\end{proof}

By combining the previous two results, we reach the following observation.

\begin{cor}
\label{PCor}
There exists an invertible matrix solution $P$ to 
\begin{equation}
\label{PGEqn}
\bar \del P = G P \quad \text{ in } Q_d
\end{equation} 
with the property that $$\norm{P}_{L^\iny\pr{Q_d}} + \norm{P^{-1}}_{L^\iny\pr{Q_d}}  \le \exp\pr{C \la}.$$
\end{cor}

\begin{lem}
If $\vec{w}$ is a solution to \eqref{vecEqn}, then $\vec{w} = P \vec{h}$, where $P$ is the invertible matrix given in Corollary \ref{PCor} and $\vec{h}$ is a 2-vector with holomorphic entries.
\end{lem}

\begin{proof}
Since $P$ is invertible, it suffices to show that $P^{-1} \vec{w}$ is a holomorphic vector.
Using equations \eqref{vecEqn} and \eqref{PGEqn}, we compute:
\begin{align*}
\bar \del \pr{P^{-1} \vec{w}} 
&= - P^{-1} \bar \del P P^{-1} \vec{w} + P^{-1} \bar \del \vec{w} \\
&= - P^{-1} A P P^{-1} \vec{w} + P^{-1} A \vec{w}
= 0,
\end{align*}
as required.
\end{proof}

\subsection{Three-ball inequality}
We now come to the three-ball inequality.
Although we used cubes for the construction of the matrix solution $P$, we now work over balls and use that $P$ and $\vec w$ are solutions in $B_d \su Q_d$.
Using that $\vec w = P \vec{h}$ and $\norm{P}_{L^\iny\pr{B_d}} \le \exp\pr{C \la}$, we have
\begin{align*}
\norm{\widetilde w_1}_{L^\iny\pr{B_{1}}} 
&= \norm{p_{11} h_1 + p_{12} h_2 }_{L^\iny\pr{B_{1}}} \\
&\le \exp\pr{C \la} \brac{ \norm{h_1 }_{L^\iny\pr{B_{1}}} +  \norm{h_2}_{L^\iny\pr{B_{1}}}} \\
&\le \exp\pr{C \la} \brac{\norm{ h_1}_{L^\iny\pr{B_{r/2}}} ^\te \norm{h_1}_{L^\iny\pr{B_{d}}} ^{1-\te} + \norm{ h_2}_{L^\iny\pr{B_{r/2}}} ^\te \norm{h_2}_{L^\iny\pr{B_{d}}} ^{1-\te}},
\end{align*}
where we have applied the Hadamard 3-circle theorem to $h_1$ and $h_2$ with $0 < r < 1 < d$ and
\begin{equation}
\label{thEst}
- \frac 1 \te 
= \frac{\log\pr{\frac{r}{2d}}} {\log {d}}
= \frac{\log r  - \log \pr{2 +  \frac 1 {F\pr{\la}}}}{\log \pr{1 + \frac 1 {2 F\pr{\la}}}}
\ge CF\pr{\la} \log r.
\end{equation}
Now, using that $\vec h = P^{-1} \vec w$ and $\norm{P^{-1}}_{L^\iny\pr{B_d}} \le \exp\pr{C \la}$, we get
\begin{align*}
& \exp\pr{- C \la} \norm{\widetilde w_1}_{L^\iny\pr{B_{1}}} \\
&\le \norm{ p^{-1}_{11} \widetilde w_1 + p^{-1}_{12} \widetilde w_2}_{L^\iny\pr{B_{r/2}}} ^\te \norm{p^{-1}_{11} \widetilde w_1 + p^{-1}_{12} \widetilde w_2}_{L^\iny\pr{B_{d}}} ^{1-\te}  \\
&+ \norm{ p^{-1}_{21} \widetilde w_1 + p^{-1}_{22} \widetilde w_2}_{L^\iny\pr{B_{r/2}}} ^\te \norm{p^{-1}_{21} \widetilde w_1 + p^{-1}_{22} \widetilde w_2}_{L^\iny\pr{B_{d}}} ^{1-\te} \\
&\le \pr{\norm{ p^{-1}_{11} \widetilde w_1 }_{L^\iny\pr{B_{r/2}}} + \norm{ p^{-1}_{12} \widetilde w_2}_{L^\iny\pr{B_{r/2}}}} ^\te \pr{\norm{p^{-1}_{11} \widetilde w_1 }_{L^\iny\pr{B_{d}}} + \norm{p^{-1}_{12} \widetilde w_2}_{L^\iny\pr{B_{d}}} }^{1-\te}  \\
&+ \pr{\norm{ p^{-1}_{21} \widetilde w_1}_{L^\iny\pr{B_{r/2}}} + \norm{p^{-1}_{22} \widetilde w_2}_{L^\iny\pr{B_{r/2}}} }^\te \pr{\norm{p^{-1}_{21} \widetilde w_1}_{L^\iny\pr{B_{d}}}+\norm{p^{-1}_{22} \widetilde w_2}_{L^\iny\pr{B_{d}}} }^{1-\te} \\
&\le 2\exp\pr{C \la} \pr{\norm{\widetilde w_1 }_{L^\iny\pr{B_{r/2}}} + \norm{\widetilde w_2}_{L^\iny\pr{B_{r/2}}}} ^\te \pr{\norm{\widetilde w_1 }_{L^\iny\pr{B_{d}}} + \norm{\widetilde w_2}_{L^\iny\pr{B_{d}}} }^{1-\te} .
\end{align*}
Recall that $\widetilde w_1 = e^{-T\pr{2 \al}} \phi u$ and since $\disp v = \frac u \phi$, then
$$\widetilde w_2 = \de^{-1} e^{-T\pr{\al - \widetilde \al}} \brac{\phi \pr{- \del_x u + i \del_y u} + u \pr{ \del_x \phi - i \del_y \phi}}.$$
It follows from Lemma \ref{grphiLem} that $\norm{\al}_{L^\iny\pr{B_{d}}} \le C_2 \la$ and $\norm{\widetilde \al}_{L^\iny\pr{B_{d}}} \le C_2 \la$.
Therefore, $ \norm{T\pr{\al - \widetilde \al}}_{L^\iny\pr{B_{d}}} \le 2 c_\iny C_2 \la$ and $\norm{T\pr{2 \al}}_{L^\iny\pr{B_{d}}} \le 2 c_\iny C_2 \la$ as well.
Using \eqref{phiBounds}, we see that
\begin{align*}
&\norm{\widetilde w_1 }_{L^\iny\pr{B_{r/2}}} \le \exp\pr{C \la} \norm{u}_{L^\iny\pr{B_{r/2}}}
\end{align*}
and a similar estimate holds in $B_{d}$.
Using the estimate \eqref{gradphiEst}, we have
\begin{align*}
\norm{\widetilde w_2}_{L^\iny\pr{B_{r/2}}}
&\le \de^{-1} \exp\pr{C \la} \brac{\norm{\phi}_{L^\iny\pr{B_{r/2}}}\norm{\gr u }_{L^\iny\pr{B_{r/2}}}  
+ \norm{u}_{L^\iny\pr{B_{r/2}}}\norm{\gr \phi}_{L^\iny\pr{B_{r/2}}}} \\
&\le \de^{-1} \exp\pr{C \la} \brac{\norm{\phi}_{L^\iny\pr{B_{r/2}}} \pr{\frac {C\la^2} r \norm{u}_{L^\iny\pr{B_{r}}}}  + \norm{u}_{L^\iny\pr{B_{r/2}}} \pr{\frac {C\la^2} r \norm{\phi}_{L^\iny\pr{B_{r}}}}} \\
&\le \de^{-1} r^{-1} \exp\pr{C \la} \norm{u}_{L^\iny\pr{B_{r}}}
\end{align*}
and 
\begin{align*}
\norm{\widetilde w_2}_{L^\iny\pr{B_{d}}}
&\le \de^{-1} \pr{b - d}^{-1} \exp\pr{C \la} \norm{u}_{L^\iny\pr{B_{b}}}
\le \de^{-1}  \exp\pr{C \la} \norm{u}_{L^\iny\pr{B_{b}}},
\end{align*}
where we have used that $b - d = \frac 1 {2 F\pr{\la}} \gtrsim \frac 1 \la$.
Observe also that 
$$\norm{u}_{L^\iny\pr{B_{1}}} \le \norm{e^{T\pr{2\al}} \phi^{-1} \widetilde w_1}_{L^\iny\pr{B_{1}}} \le \exp\pr{C \la}\norm{\widetilde w_1}_{L^\iny\pr{B_{1}}}.$$
Combining our observations, we have
\begin{align*}
\norm{u}_{L^\iny\pr{B_{1}}}
&\le \exp\pr{C \la} \brac{\norm{u }_{L^\iny\pr{B_{r/2}}} + \de^{-1} r^{-1} \norm{u}_{L^\iny\pr{B_{r}}}} ^\te \brac{\norm{u}_{L^\iny\pr{B_{d}}} + \de^{-1} \norm{u}_{L^\iny\pr{B_{b}}} }^{1-\te} \\
&\le \de^{-1} \exp\pr{C \la} \pr{ r^{-1} \norm{u}_{L^\iny\pr{B_{r}}}}^\te \norm{u}_{L^\iny\pr{B_{b}}}^{1-\te} \\
&\le \de^{-1} \exp\brac{\pr{C+C_0}  \la} \pr{ r^{-1} \norm{u}_{L^\iny\pr{B_{r}}}}^\te,
\end{align*}
where we have applied \eqref{localBd}.
Using \eqref{deBd}, \eqref{localNorm}, and \eqref{thEst}, it follows that
\begin{align*}
\norm{u}_{L^\iny\pr{B_{r}}}
&\ge r \frac{c \sqrt \la}{\log \la}\exp\brac{- \frac{c_1 \la^p + \pr{C + C_0 + m} \la }{\te}} \\
&\ge r \exp\brac{C F\pr{\la} \la^q \log r}
\ge r^{ C \la^q F\pr{\la}},
\end{align*}
as required.

\section{The proof of Theorem \ref{LandisThm}}
\label{Landis}

We begin with a proposition that serves as the main tool in the iteration scheme.

\begin{prop}
Assume that $V : \R^2 \to \R$ satisfies \eqref{V+Cond} and \eqref{V-Cond}.
Let $u: \R^2 \to \R$ be a solution to \eqref{ePDE} for which \eqref{uBd} holds.
Let $\eps \in \pr{0, \frac{\eps_0}{1 + \eps_0}}$.
Suppose that for any $S \ge \tilde S\pr{C_0, c_0, \eps_0, \eps}$, there exists an $\al \in \pb{1, 2}$ so that
\begin{align}
\inf_{\abs{z_0} = S}\norm{u}_{L^\iny\pr{B_1\pr{z_0}}} \ge \exp\pr{- S^\al}.
\label{stepnEst}
\end{align}
Set $R = S + \pr{\frac S 2}^{\frac1{1- \eps}} -1$.
\begin{enumerate}
\item If $\al > \frac{1}{1 - \eps}$, then with $\be = \al - \frac{\al -1}{2}\eps$, it holds that
\begin{equation}
\inf_{\abs{z_1} = R} \norm{u}_{L^\iny\pr{B_1\pr{z_1}}} \ge \exp\pr{- R^\be}.
\label{stepn1Est}
\end{equation}
\item If $\al \in \pb{1, \frac{1}{1 - \eps} }$, then 
\begin{equation}
\inf_{\abs{z_1} = R} \norm{u}_{L^\iny\pr{B_1\pr{z_1}}} \ge \exp\pr{-C R^{1 + \eps} \log R},
\label{stepNEst}
\end{equation}
where $C$ depends on $C_0$.
\end{enumerate}
\label{itProp}
\end{prop}

\begin{proof}
Define $T = \pr{\frac S 2}^{\frac 1{1 - \eps}}$ and set $b = 1 + \frac{S}{2T}$. 
Let $z_1 \in \R^2$ be such that $\abs{z_1} = S + T -1= R$.
Define 
\begin{align*}
&\tilde u\pr{z} = u\pr{z_1 + T z} \\
&\tilde V\pr{z} = T^2 V\pr{z_1 + T z}.
\end{align*}
Then $\LP \tilde u - \tilde V \tilde u = 0$ in $Q_b$.
Assumption \eqref{V+Cond}  implies that $\norm{\tilde V_+}_{L^\iny\pr{Q_b}} \le T^2$ while condition \eqref{V-Cond} gives that $\norm{\tilde V_-}_{L^\iny\pr{Q_b}} \le T^2 \exp\brac{- c_0 \pr{\frac S 2 - 1}^{1 + \eps_0}}$.
Moreover, $\norm{\tilde u}_{L^\iny\pr{B_b}} \le \exp\brac{C_0 \pr{\frac 3 2 S + 2T}} \le \exp\pr{5 C_0 T}$ and from \eqref{stepnEst} we see that with $z_0 := S \frac{z_1}{\abs{z_1}}$, $\norm{\tilde u}_{L^\iny\pr{B_1}} \ge \norm{u}_{L^\iny\pr{B_1\pr{z_0}}} \ge \exp\pr{- S^\al }$.

With $\la = T$, we see that $b = 1 + \la^{- \eps}$ and $\norm{\tilde V_+}_{L^\iny\pr{Q_b}} \le \la^2$.
Furthermore, if $S$ is sufficiently large in the sense that $\frac{\pr{S/2-1}^{1 + \eps_0}}{\pr{S/2}^{\frac {1}{1 -\eps}}} \ge \frac{3m}{c_0}$ (which is always possible because of the relationship between $\eps$ and $\eps_0$), then we have $\norm{\tilde V_-}_{L^\iny\pr{Q_b}} \le \de^2$ where $\de$ is given by \eqref{deBd} with $c$ depending only on $m$.
With $C_1 = 5 C_0$, we see that $\norm{\tilde u}_{L^\iny\pr{B_b}} \le \exp\pr{C_1 \la}$.
Finally, setting $c_1 = 4 \ge 2^\al$ and $p = \al\pr{1 - \eps}$, we have $\norm{\tilde u}_{L^\iny\pr{B_1}} \ge \exp\pr{- c_1 \la^p}$.
Now we may apply Theorem \ref{OofV} to conclude that for $r < 1$,
$$\norm{\tilde u}_{L^\iny\pr{B_r}} \ge r^{C \la^q},$$
where $q = \max\set{p, 1} + \eps$ and $C$ depends on $C_0$ and $m$.
Choosing $r = T^{-1} = \la^{-1}$, we see that
$$\norm{u}_{L^\iny\pr{B_1\pr{z_1}}} \ge \exp\pr{- C \la^q \log \la} = \exp\pr{- C T^q \log T}.$$ %\ge \exp\pr{- R^\be}.$$

If $\al > \frac{1}{1 - \eps}$, then $q = p + \eps =\al - \pr{\al-1} \eps \in \pr{1 + \eps, \al}$.
If $S$ is sufficiently large in the sense that 
$$\frac{\pr{S/2}^{ \frac{\eps^2}{2\pr{1-\eps}^2}}}{\log\pr{S/ 2}} \ge \frac C {1 - \eps},$$
then $R^\be \ge C T^q \log T$ and it follows that
$$\norm{u}_{L^\iny\pr{B_1\pr{z_1}}} \ge \exp\pr{- R^\be}.$$
Since $z_1 \in \R^2$ with $\abs{z_1} = R$ was arbitrary, \eqref{stepn1Est} has been proved.

On the other hand, if $\al \in \pb{1, \frac{1}{1 - \eps}}$, then $p \le 1$ so that $q = 1 + \eps$.
Since $R \ge T$, it follows that
$$\norm{u}_{L^\iny\pr{B_1\pr{z_1}}} \ge \exp\pr{- C R^{1+\eps} \log R}.$$
Again, because $z_1 \in \R^2$ with $\abs{z_1} = R$ was arbitrary, \eqref{stepNEst} follows.
\end{proof}

Now we present the proof of the main theorem.

\begin{proof}[Proof of Theorem \ref{LandisThm}]
%Let $\eps \in \pr{0, \frac{\eps_0}{1 + \eps_0}}$.
Let $\eps > 0$ be given and set $\eps_1 = \frac \eps 2$.
Since $\norm{V}_{L^\iny\pr{\R^2}} \le 1$, then Lemma 3.10 in \cite{BK05}, for example, implies that if $\abs{z_0} \ge 1$, then
\begin{align*}
\inf_{\abs{z_0} = S_0}\norm{u}_{L^\iny\pr{B_1\pr{z_0}}} \ge \exp\pr{- c S_0^{4/3} \log S_0},
\end{align*}
where $c$ depends on $C_0$.
We choose $\al_0 \in \pb{4/3, 2}$ so that $c \tilde S^{4/3} \log \tilde S \le \tilde S^{\al_0}$, where $\tilde S\pr{C_0, c_0, \eps_0, \eps_1}$ is the lower bound on $S$ given in Proposition \ref{itProp}.
For any $S_0 \ge \tilde S$, we see that
\begin{align*}
\inf_{\abs{z_0} = S_0}\norm{u}_{L^\iny\pr{B_1\pr{z_0}}} \ge \exp\pr{- S_0^{\al_0}}.
\end{align*}
Assume that $\al_0 > \frac{1}{1 - \eps_1}$.
%For $\al$, $\be$ given in Proposition \ref{itProp}, we observe that $\frac \be \al < 1 - \frac{1}{2} \eps^2$.
For $n = 0, 1, 2, \ldots$, define $\al_{n+1} = \al_n - \frac{\al_n -1} 2 \eps_1$ and observe that as long as $\al_n >  \frac{1}{1 - \eps_1}$, $\frac {\al_{n+1}} {\al_n} < 1 - \frac{\eps_1^2}{2}$.
Therefore, there exists $N \in \N$ such that $\al_n >  \frac{1}{1 - \eps_1}$ for all $n = 0, 1, \ldots, N-1$, while $\al_{N} \le  \frac{1}{1 - \eps_1}$.
For each $n = 0, 1, 2, \ldots N$, we also define $S_{n+1} = S_n + \pr{\frac{S_n} 2}^{\frac 1 {1 -\eps_1}} - 1$.
Since $\al_n >  \frac{1}{1 - \eps_1}$ for each $n = 1, 2, \ldots, N$-1, then applications of the first case of Proposition \ref{itProp} with $\eps = \eps_1$, $\al = \al_{n}$, and $S = S_{n}$ give
\begin{align*}
\inf_{\abs{z_{n+1}} = S_{n+1}}\norm{u}_{L^\iny\pr{B_1\pr{z_{n+1}}}} \ge \exp\pr{- S_{n+1}^{\al_{n+1}}}.
\end{align*}
That is, Proposition \ref{itProp} holds with $\be = \al_{n+1}$ and $R = S_{n+1}$.
In particular,
\begin{align*}
\inf_{\abs{z_{N}} = S_{N}}\norm{u}_{L^\iny\pr{B_1\pr{z_{N}}}} \ge \exp\pr{- S_{N}^{\al_{N}}}.
\end{align*}
Since $\al_N \le  \frac{1}{1 - \eps_1}$, another application of Proposition \ref{itProp} (this time using the second case) shows that 
\begin{align*}
\inf_{\abs{z_{N+1}} 
= S_{N+1}}\norm{u}_{L^\iny\pr{B_1\pr{z_{N+1}}}} 
\ge \exp\pr{- C S_{N+1}^{1 +\eps_1}\log S_{N+1}}
\ge \exp\pr{- S_{N+1}^{1 + \eps}},
\end{align*}
completing the proof.
\end{proof}

\section{The proof of Proposition \ref{matrixEstProp}}
\label{P1Proof}

We now prove Proposition \ref{matrixEstProp} by following the argument in \cite{BR03}.
We start from Cartan's Lemma, as given by Malgrange:

\begin{thm}[Theorem 2 from Chapter 9 of \cite{Mal58}]
\label{CarLem}
Let $K$ be a rectangle in $\C$, and let $L, M$ be compact sets in $\C^\ell$, $\C^m$, respectively.
Let $H = K \cap \set{\Re z = 0}$.
Let $C\pr{z, \la, \mu}$ be a $C^\iny$ function in a neighborhood of $H \times L \times M$ that is holomorphic in $z$ and $\la$ with values in $GL\pr{m, \C}$.
Let $K_1 = K \cap \set{z \in \C : \Re z \ge 0}$ and $K_1 = K \cap \set{z \in \C : \Re z \le 0}$.
Then there exists function $C_1\pr{z, \la, \mu}$ and $C_2\pr{z, \la, \mu}$ in neighborhoods of $K_1 \times L \times M$ and $K_2 \times L \times M$, respectively, satisfying the same regularity conditions as $C$ and such that, in a neighborhood of $H \times L \times M$, $C = C_1 C_2^{-1}$.
\end{thm}

Repeated applications of this theorem (with $L, M = \emptyset$) produce a collection of analytic functions $\set{\ga_i}_{i=1}^{i_0}$, where each $\ga_i$ is defined on $U_i$ and satisfies $H_i = \ga_{i-1} \ga_i^{-1}$ on $U_{i-1} \cap U_i$.
As given, these are no explicit bounds for these functions $\ga_i$, so our goal is to produce such estimates.
To do this, we find an invertible analytic function $h$ defined on $R$ and then set 
\begin{equation}
\label{giDef}
g_i = \ga_i h \; \text{ on } \; U_i.
\end{equation}
Then $g_{i-1} g_i^{-1} = \ga_{i-1} h h^{-1} \ga_i^{-1} = \ga_{i-1} \ga_i^{-1} = H_i$ on $U_{i-1} \cap U_i$, as desired.
%Before we describe the construction of $h$, we establish some important bounds for $\ga_i$ and $g_i$.

To find $h$ and establish that both $h$ and $g_i$ have good bounds, we rely on the Wiener-Masani Theorem.
The following statement is from \cite{BR03}, see also \cite{WM57}.
We use this theorem over a rectangle instead of a ball.

\begin{thm}[\cite{BR03} Theorem 2.1]
\label{WMThm}
Let $A_0$ be a positive definite $N \times N$ matrix of smooth functions defined on the circle.
Then there exists a $N \times N$ matrix $h$ of holomorphic functions in the disk, extending smoothly to the boundary such that
$$A_0 = h^* h$$
on the circle, and such that $g = h^{-1}$ is also holomorphic in the disk and extends smoothly to the boundary.
The matrix $h$ is uniquely determined up to multiplication from the left by a constant unitary matrix.
 \end{thm}

Thus, we need to prescribe the values of $\pr{h^{-1}}^* h^{-1}$ on $\del R$.
Define the sets
\begin{equation}
\label{WiDef}
W_i = \left\{\begin{array}{ll} 
U_i \setminus \pr{U_i \cap U_{i+1}} & \text{ if } i = 0 \\
U_i \setminus \brac{\pr{U_{i-1} \cap U_i} \cup \pr{U_i \cap U_{i+1}}} & \text{ if } i = 1, \ldots, i_0 - 1 \\
U_i \setminus \pr{U_{i-1} \cap U_i} & \text{ if } i = i_0 \end{array}\right.
\end{equation}
First define $h$ on each $\del R \cap W_i$ so that $\pr{h^{-1}}^* h^{-1} = \ga_i^* \ga_i$ there.
This implies that $g_i^* g_i = I$ on this part of the boundary.
Then on each $\del R \cap \pr{U_{i-1} \cap U_i}$, the function $\pr{h^{-1}}^* h^{-1}$is defined as a convex combination of $\ga_{i-1}^* \ga_{i-1}$ and $\ga_i^* \ga_i$.
%And on $\del R \cap \pr{U_{i} \cap U_{i+1}}$, $\pr{h^{-1}}^* h^{-1}$ is similarly defined in terms of $\ga_{i}^* \ga_{i}$ and $\ga_{i+1}^* \ga_{i+1}$.
Once this process has been carried out, we have that $\pr{h^{-1}}^* h^{-1}$ is defined unambiguously on $\del R$ and an application of the Weiner-Masani Theorem implies that there exists an analytic function function $h^{-1}$ defined in $R$.
In conclusion, the required analytic function $h$ exists.

Once we establish that \eqref{gBds} holds, the proof of Proposition \ref{matrixEstProp} is complete.
Now we work to establish bounds for $\ga_i$ and $g_i$ through a series of technical results.

\begin{lem}
\label{gaComp}
On $U_{i-1} \cap U_i$,
$$\frac 1{10} \le \frac{\abs{\ga_{i-1}}}{\abs{\ga_i}} \le 10 \; \text{ and } \; \frac 1{10} \le \frac{\abs{\ga_{i-1}^{-1}}}{\abs{\ga_i^{-1}}} \le 10.$$
\end{lem}

\begin{proof}
Since $H_i = \ga_{i-1}\ga_i^{-1}$ on $U_{i-1} \cap U_i$ and we can write $\ga_{i-1} = \ga_{i-1} \ga_i^{-1}\ga_i = H_i \ga_i$, then 
$$\abs{\ga_{i-1}} \le \norm{H_i} \abs{\ga_i} \le 10 \abs{\ga_i},$$
from the assumed bound on $H_i$.
Similarly,
$$\abs{\ga_{i}} \le \norm{\ga_{i} \ga_{i-1}^{-1}} \abs{\ga_{i-1}} = \norm{H_i^{-1}} \abs{\ga_{i-1}} \le 10 \abs{\ga_{i-1}}.$$
Combining these two bounds leads to the first stated estimate.
The same argument for the inverses gives the second estimate.
\end{proof}

\begin{lem}
\label{gComp}
On $U_{i-1} \cap U_i$,
$$\frac 1{10} \le \frac{\abs{g_{i-1}}}{\abs{g_i}} \le 10 \; \text{ and } \; \frac 1{10} \le \frac{\abs{g_{i-1}^{-1}}}{\abs{g_i^{-1}}} \le 10.$$
\end{lem}

\begin{proof}
We have
\begin{align*}
\abs{g_{i-1}}^2
&= \tr\pr{g_{i-1}^* g_{i-1}}
= \tr\pr{g_{i-1} g_{i-1}^*}
= \tr\pr{\ga_{i-1} h h^* \ga_{i-1}^*} \\
&= \tr\pr{\ga_{i-1} \ga_i^{-1}\pr{\ga_i h h^* \ga_i^*} \ga_i^{-1*} \ga_{i-1}^*}
= \abs{\ga_{i-1} \ga_i^{-1} \ga_i h}^2
\le \norm{\ga_{i-1} \ga_i^{-1}}^2 \abs{ \ga_i h}^2 \\
&= \norm{H_i}^2 \abs{ g_i}^2
\le 10^2 \abs{ g_i}^2.
\end{align*}
Similarly,
\begin{align*}
\abs{g_{i}}^2
&%= \tr\pr{g_{i} g_{i}^*}
= \tr\pr{\ga_{i} h h^* \ga_{i}^*} 
= \tr\pr{\ga_{i} \ga_{i-1}^{-1}\pr{\ga_{i-1} h h^* \ga_{i-1}^*} \ga_{i-1}^{-1*} \ga_{i}^*}
%= \abs{\ga_{i} \ga_{i-1}^{-1} \ga_{i-1} h}^2
\le \norm{\ga_{i} \ga_{i-1}^{-1}}^2 \abs{ \ga_{i-1} h}^2 \\
&= \norm{H_i^{-1}}^2 \abs{ g_{i-1}}^2
\le 10^2 \abs{ g_{i-1}}^2.
\end{align*}
Combining these two observations leads to the first bound on $U_{i-1} \cap U_i$, and the same bounds hold for the inverses.
\end{proof}

\begin{lem}
\label{gEst}
On $\del R \cap U_i$,
$$\frac 2 {10^2} \le \abs{g_i}^2 \le 2 \cdot 10^2 \; \text{ and } \; \frac 2 {10^2} \le \abs{g_i^{-1}}^2 \le 2 \cdot 10^2.$$
\end{lem}

\begin{proof}
Since $g_i^* g_i = I$ on $\del R \cap W_i$ by construction, then $\abs{g_i}^2 = 2$ there.
On $\del R \cap \pr{U_{i-1} \cap U_i}$, we define $\pr{h^{-1}}^* h^{-1} = \te \ga_{i-1}^* \ga_{i-1} + \pr{1 - \te} \ga_i^* \ga_i$ for some $0 \le \te \le 1$, from which it follows that
$$I =  \te h^* \ga_{i-1}^* \ga_{i-1}h + \pr{1 - \te} h^* \ga_i^* \ga_i h.$$
Then
\begin{align*}
2 &= \tr\pr{I}
=  \te \tr\pr{h^* \ga_{i-1}^* \ga_{i-1}h} + \pr{1 - \te} \tr\pr{h^* \ga_i^* \ga_i h} \\
&\le \te 10^2 \tr\pr{h^* \ga_i^* \ga_i h} +\pr{1 - \te} \tr\pr{h^* \ga_i^* \ga_i h} ,
\end{align*}
where we have used the idea from the proof of Lemma \ref{gComp}.
Therefore, $\frac 2 {10^2} \le \abs{g_i}^2$ on $\del R \cap \pr{U_{i-1} \cap U_i}$.
And since
\begin{align*}
2 &= \tr\pr{I}
=  \te \tr\pr{h^* \ga_{i-1}^* \ga_{i-1}h} + \pr{1 - \te} \tr\pr{h^* \ga_i^* \ga_i h} \\
&\ge \te 10^{-2} \tr\pr{h^* \ga_i^* \ga_i h} +\pr{1 - \te} \tr\pr{h^* \ga_i^* \ga_i h} ,
\end{align*}
then $\frac 2 {10^2} \le \abs{g_i}^2 \le 2 \cdot 10^2$.
On $\del R \cap \pr{U_i \cap U_{i+1}}$, we can similarly show that $\frac 2 {10^2} \le \abs{g_i}^2 \le 2 \cdot 10^2$.
Combining these three bounds leads to the first estimate in the conclusion of the lemma.
An analogous argument shows that each $\abs{g_i^{-1}}$ satisfies the same bounds.
\end{proof}

To get interior bounds for $\abs{g_i}^2$ on $U_i$, we define and use a subharmonic function $v$.

\begin{lem}
\label{vDefLem}
For $i = 0, \ldots, i_0$, set
\begin{align*}
& c_i^+ = c_{i+1}^- = i \frac A \de \\
%& c_{i}^- = \pr{i-1} \frac A \de \\
& b_i^+ = b_{i+1}^- = -\frac{i\pr{i+1}}{2} A - i B, % \\
%&  b_{i}^- = -\frac{i\pr{i-1}}{2} A - \pr{i-1} B,
\end{align*}
where $A = 10.5 \log 10$ and $B = 3 \log 10$.
Then the function defined piecewise by
\begin{equation}
\label{vDef}
v = \left\{ \begin{array}{ll}
%\abs{g_{i-1}}^2 e^{c_i^- x + b_i^-} & \text{ on } W_{i-1} \\
\max\set{\abs{g_{i-1}}^2 e^{c_i^- x + b_i^-}, \abs{g_{i}}^2 e^{c_i^+ x + b_i^+}} & \text{ on } U_{i-1} \cap U_i \; \text{ for }\, i= 1, \ldots, i_0\\
\abs{g_{i}}^2 e^{c_i^+ x + b_i^+} & \text{ on } W_{i}  \; \text{ for }\, i = 0, \ldots, i_0
\end{array}\right.
\end{equation}
is continuous and subharmonic on $R$.
\end{lem}

\begin{proof}
Recall that if $f = e^\phi$, where $\phi$ is continuous and subharmonic, then so too is $f$.
Since $\log \abs{g}$ is subharmonic whenever $g$ is analytic and $cx + b$ is harmonic, then for any analytic $g$, $\log\pr{\abs{g}^2 e^{c x + b}} = 2 \log \abs{g} + c x + b$ is continuous and subharmonic.
Since each $g_i$ is analytic, then every function used to define $v$ is subharmonic.
In particular, $v$ is continuous and subharmonic on each $W_i$.
Moreover, since the maximum of two continuous subharmonic functions is the same, then $v$ is also continuous and subharmonic on each $U_{i-1} \cap U_i$.
It remains to show that we have compatibility along the boundaries of each $W_i$ and $U_{i-1} \cap U_i$.

For $i = 1, \ldots, i_0$, set $x_i = i \de + \frac 1 4 \de$, $x_i^- = i \de$, and $x_i^+ = i \de + \frac 1 2 \de$.
If we additionally define $x_0^+ = 0$ and $x_{i_0+1}^- = 1$, note that 
$$U_{i-1} \cap U_i = \brac{x_i^-, x_i^+} \times \brac{0, 1}$$ 
and 
$$W_i = \brac{x_i^+, x_{i+1}^-} \times \brac{0, 1}.$$
If $x$ is near $x_i^-$, then $x$ is near $W_{i-1} \cap \pr{U_{i-1} \cap U_i}$ and since $c_{i-1}^+ = c_{i}^-$ and $b_{i-1}^+= b_{i}^-$, then we want $v\pr{x, y} = \abs{g_{i-1}}^2 e^{c_i^- x + b_i^-}$ in this region.
Similarly, if $x$ is near $x_i^+$, then we need to show that $v\pr{x, y} = \abs{g_i}^2 e^{c_i^+ x + b_i^+}$.

By Lemma \ref{gComp}
\begin{align*}
\abs{g_{i}}^2 e^{c_i^+ x + b_i^+}
&\le 10^2 \abs{g_{i-1}}^2 e^{c_i^+ x + b_i^+}
=  \abs{g_{i-1}}^2 e^{c_i^- x + b_i^-} 10^2 e^{\pr{c_i^+ - c_i^-} x + \pr{b_i^+- b_i^-} }. % \\
\end{align*}
If $x \in [x_i^-, x_i^- + \eps \de)$ for some $\eps > 0$, then
\begin{align*}
\pr{c_i^+ - c_i^-} x + \pr{b_i^+- b_i^-}
&= \frac A \de x - \pr{B + i A}
< \frac A \de \pr{i \de + \eps \de} - \pr{B + i A}  
= A \eps - B.
\end{align*}
Assuming that $\eps \le \frac 1 {10.5}$, we have $10^2 e^{ A \eps - B} = 10^2 e^{ \pr{10.5 \eps - 3} \log 10} \le 1$
%\begin{align*}
%10^2 e^{ A \eps - B}
%&= 10^2 e^{ \pr{10.5 \eps - 3} \log 10}
%\le 1
%\end{align*}
and we conclude that 
\begin{align*}
\abs{g_{i}}^2 e^{c_i^+ x + b_i^+}
&<  \abs{g_{i-1}}^2 e^{c_i^- x + b_i^-}.
\end{align*}
Therefore, when $x \in [x_i^-, x_i^- + \frac 1 {10.5} \de)$, $\max\set{\abs{g_{i-1}}^2 e^{c_i^- x + b_i^-}, \abs{g_{i}}^2 e^{c_i^+ x + b_i^+}} = \abs{g_{i-1}}^2 e^{c_i^- x + b_i^-}$, proving that $v$ is continuous along $W_{i-1} \cap \pr{U_{i-1} \cap U_i}$.

We repeat the argument near the boundary of $W_{i} \cap \pr{U_{i-1} \cap U_i}$.
By Lemma \ref{gComp}
\begin{align*}
\abs{g_{i-1}}^2 e^{c_i^- x + b_i^-}
&\le 10^2 \abs{g_{i}}^2 e^{c_i^- x + b_i^-}
=  \abs{g_{i}}^2 e^{c_i^+ x + b_i^+} 10^2 e^{-\brac{\pr{c_i^+ - c_i^-} x + \pr{b_i^+- b_i^-} }}. % \\
\end{align*}
If $x \in (x_i^+ - \eps \de, x_i^+]$ for some $\eps > 0$, then
\begin{align*}
\pr{c_i^+ - c_i^-} x + \pr{b_i^+- b_i^-}
&= \frac A \de x - \pr{B + i A}
> \frac A \de \pr{i \de + \frac 1 2 \de - \eps \de} - \pr{B + i A}  
= \frac A 2 - A \eps - B.
\end{align*}
Assuming that $\eps \le \frac{11} {21}$, we have $10^2 e^{ A \eps + B - \frac A 2} = 10^2 e^{ \pr{10.5 \eps + 3 - 10.5} \log 10} \le 1$ and we conclude that 
\begin{align*}
\abs{g_{i}}^2 e^{c_i^+ x + b_i^+}
&<  \abs{g_{i-1}}^2 e^{c_i^- x + b_i^-}.
\end{align*}
Therefore, when $x \in (x_i^+ - \frac{11} {21} \de, x_i^+]$, $\max\set{\abs{g_{i-1}}^2 e^{c_i^- x + b_i^-}, \abs{g_{i}}^2 e^{c_i^+ x + b_i^+}} = \abs{g_{i}}^2 e^{c_i^+ x + b_i^+} $, proving that $v$ is continuous along $W_{i} \cap \pr{U_{i-1} \cap U_i}$ as well and completing the proof.
\end{proof}

Now we'll use the maximum principle to estimate $v$ in terms of $\disp \max_{\del R}v$.

\begin{lem}
\label{vBd}
For the function $v$ as defined in \eqref{vDef}, there exists a universal constant $C > 0$ so that $v \le C e^{C/\de^2}$ in $R$.
\end{lem}

\begin{proof}
We start with $\del R \cap W_i$.
Since $v = \abs{g_i}^2 e^{c_i^+ x + b_i^+}$ on $W_i$, and $\abs{g_i}^2 = 2$ on $\del R \cap W_i$, then 
$$v = 2 e^{i \frac A \de x - \frac{i\pr{i+1}}{2} A - iB} \le 2 e^{i \frac A \de x_{i+1}^- - \frac{i\pr{i+1}}{2} A - iB} = 2 e^{ \frac{i\pr{i+1}}{2} A - iB} \; \text{ on } \; \del R \cap W_i.$$
Taking a supremum over $i \in \set{0, \ldots, i_0}$, we see that
\begin{equation}
\label{WBdyEst}
v \le 2 e^{ \frac{i_0\pr{i_0+1}}{2} A - i_0 B} \; \text{ on } \; \del R \cap \pr{\bigcup_{i=0}^{i_0}W_i}.
\end{equation}
Next we examine $v$ on $\del R \cap \pr{U_{i-1} \cap U_i}$.
There we have $v = \max\set{\abs{g_{i-1}}^2 e^{c_i^- x + b_i^-}, \abs{g_{i}}^2 e^{c_i^+ x + b_i^+}}$ and $x_i^- \le x \le x_i^+$.
By Lemma \ref{gEst}, $\abs{g_{i-1}}^2 \le 2 \cdot 10^2$ and $\abs{g_{i}}^2 \le 2 \cdot 10^2$ in $U_{i-1} \cap U_i$.
Examining the exponentials, we get 
\begin{align*}
\max\set{e^{c_i^- x + b_i^-}, e^{c_i^+ x + b_i^+}} 
&\le \max\set{e^{c_i^- x_i^+ + b_i^-}, e^{c_{i+1}^- x_i^+ + b_{i+1}^-}} \\
%&= \max\set{e^{\pr{i-1} \frac A \de \pr{i\de + \frac \de 2} - \pr{\frac{i\pr{i-1}}{2} A + \pr{i-1} B}}, e^{i \frac A \de \pr{i\de + \frac \de 2} - \pr{\frac{i\pr{i+1}}{2} A + i B}}} \\
&= \max\set{e^{ \frac{i^2-1} 2A  - \pr{i-1} B}, e^{ \frac{i^2} 2 A  -  i B}}
= e^{ \frac{i^2} 2 A  -  i B},
\end{align*}
since $A = 10.5 \log 10$ and $B = 3 \log 10$.
Therefore,
$$v \le 2 \cdot 10^2 e^{ \frac{i^2} 2 A  -  i B} \; \text{ on } \; \del R \cap \pr{U_{i-1} \cap U_i}.$$
Taking a supremum over $i \in \set{1, \ldots, i_0}$, we see that
\begin{equation}
\label{UBdyEst}
v \le 2 \cdot 10^2 e^{ \frac{i_0^2} 2 A  -  i_0 B} \; \text{ on } \; \del R \cap \brac{\bigcup_{i=1}^{i_0} \pr{U_{i-1} \cap U_i}}.
\end{equation}
Combining \eqref{WBdyEst} and \eqref{UBdyEst} shows that $v \le 2 e^{ \frac{i_0\pr{i_0+1}}{2} A - i_0 B}$ on $\del R$.
Since $i_0 = \frac 1 \de - \frac 3 2$, then the conclusion of the lemma follows from an application of the maximum principle.
\end{proof}

Now we have all of the preliminary results required to prove Proposition \ref{matrixEstProp}.

\begin{proof}[Proof of Proposition \ref{matrixEstProp}]
By Theorem \ref{CarLem}, there exists a collection $\set{\ga_i}_{i=1}^{i_0}$ of analytic function, where each $\ga_i$ is defined on $U_i$ and satisfies $H_i = \ga_{i-1} \ga_i^{-1}$ on $U_{i-1} \cap U_i$.

On each $\del R \cap W_i$, define $h$  so that $\pr{h^{-1}}^* h^{-1} = \ga_i^* \ga_i$ there.
%This implies that $g_i^* g_i = I$ on this part of the boundary.
Then on each $\del R \cap \pr{U_{i-1} \cap U_i}$, define $\pr{h^{-1}}^* h^{-1}$ to be a convex combination of $\ga_{i-1}^* \ga_{i-1}$ and $\ga_i^* \ga_i$.
Since $\pr{h^{-1}}^* h^{-1}$ is defined unambiguously on $\del R$, an application of Theorem \ref{WMThm} implies that there exists an analytic function function $h^{-1}$ defined in $R$.

On each $U_i$, define $g_i = \ga_i h$.
Then each $g_i$ is defined and analytic on $U_i$ with $g_{i-1} g_i^{-1} = \ga_{i-1} h h^{-1} \ga_i^{-1} = \ga_{i-1} \ga_i^{-1} = H_i$ on $U_{i-1} \cap U_i$.

For any $i \in  \set{0, 1, \ldots, i_0}$, $\abs{g_i}^2 e^{c_i^+ x + b_i^+} \le v$ on $U_i$ and it follows from Lemma \ref{vBd} and the definition of $v$ as given in \eqref{vDef} that 
$$\abs{g_i\pr{x,y}}^2 \le v\pr{x, y} e^{- c_i^+ x - b_i^+} \le v\pr{x, y} e^{- c_{i+1}^- x - b_{i+1}^-} \le C e^{C/\de^2} \; \text{ in } U_i.$$
A similar argument may be made for $g_i^{-1}$, completing the proof.
\end{proof}

\def\cprime{$'$}

\end{document}